\documentclass[dvipdf]{amsart}

\title[]{A rigidity theorem for holomorphic disks in Teichm\"uller space}
\author{Hideki Miyachi}
\address{Department of Mathematics,
Graduate School of Science,
Osaka University,
Machikaneyama 1-1, Toyonaka, Osaka 560-0043, Japan}
\dedicatory{This paper is dedicated to Professor Hiroshige Shiga on the occasion of his 60th birthday.}
\thanks{The author is partially supported by the Ministry of Education, Science, Sports and Culture, Grant-in-Aid for Scientific Research (C), 21540177.}
\subjclass[2010]{Primary 32G15; Secondary 32F10, 32T15, 32E35}
\keywords{Teichm\"uller space, Teichm\"uller distance, Kobayashi distance, Pseudoconvex domains}

\makeatletter
\@addtoreset{equation}{section}

\makeatother

\usepackage{amsfonts, amsmath, amssymb, amsthm}
\usepackage{url}
\usepackage[dvipdf]{graphicx}
\usepackage[dvipdf]{color}

\newtheorem{theorem}{Theorem}[section]

\newtheorem{corollary}{Corollary}[section]

\newcommand{\gromov}[3]{\langle #1\,|\,#2\rangle_{#3}}
\newcommand{\Teichb}[1]{\mathcal{T}(#1)}
\newcommand{\ext}{{\rm Ext}}
\newcommand{\proj}{{\rm proj}}

\newcommand{\cl}[1]{{\rm cl}_{GM}(#1)}
\newcommand{\partialGM}[1]{{\partial_{GM}#1}}

\begin{document}
\maketitle

\begin{abstract}
In this paper,
we discuss a rigidity property for holomorphic disks in Teichm\"uller space.
In fact,
we give an improvement of Tanigawa's rigidity theorem.
We will also treat the rigidity property of holomorphic disks for complex manifolds.
We observe the rigidity property is valid for bounded strictly 
pseudoconvex domains with $C^{2}$-boundaries,
but the rigidity property does not hold for product manifolds.
\end{abstract}

\section{Introduction}
\subsection{}
Let $S$ be a compact orientable surface with negative Euler characteristic
(possibly with boundary).
Let $\Teichb{S}$ be the Teichm\"uller space of $S$
and $d_{T}$ denotes
the Teichm\"uller distance on $\Teichb{S}$.
Fix $x_{0}\in \Teichb{S}$.
The \emph{Gromov product} with basepoint $x_{0}$ is defined by
$$
\gromov{x}{y}{x_{0}}=\frac{1}{2}(d_{T}(x_{0},x)+d_{T}(x_{0},y)-d_{T}(x,y))
$$
for $x,y\in \Teichb{S}$.
The main purpose of this paper is to show the following.

\begin{theorem}[Rigidity of holomorphic disks]
\label{thm:main}
Let $f_{1}$ and $f_{2}$ be holomorphic mappings
from the unit disk $\mathbb{D}$ to $\Teichb{S}$.
Suppose that there is a measurable set $E\subset \partial \mathbb{D}$
of positive linear measure with the following property:
For any $z_{0}\in E$,
there is a sequence $\{z_{n}\}_{n=1}^{\infty}\subset \mathbb{D}$
such that $z_{n}\to z_{0}$ nontangentially
and $\gromov{f_{1}(z_{n})}{f_{2}(z_{n})}{x_{0}}\to \infty$.
Then,
$f_{1}(z)=f_{2}(z)$ for all $z\in \mathbb{D}$.
\end{theorem}

We say here that a sequence in $\mathbb{D}$
converges to $z_{0}\in \partial \mathbb{D}$ \emph{nontangentially}
if it tends to $z_{0}$ from the inside of any fixed Stolz region with the vertex at $z_{0}$
(cf. \cite{Tsuji}).

Since $|\gromov{x}{y}{x_{0}}-\gromov{x}{y}{x_{1}}|\le
d_{T}(x_{0},x_{1})$
for $x,y,x_{0},x_{1}\in \Teichb{S}$,
the assumption in the theorem is independent of the choice of the basepoint.
Furthermore,
since $\gromov{x}{y}{x_{0}}\le d_{T}(x_{0},x)$,
each holomorphic mapping $f_{i}$ ($i=1,2$)
in the theorem satisfies $d_{T}(x_{0},f_{i}(z_{n}))\to \infty$ as $n\to \infty$.

\subsection{}
A typical example of a pair of holomorphic mappings
satisfying the assumption in Theorem \ref{thm:main}
is a pair consisting of 
$f_{1},f_{2}\colon \mathbb{D}\to \Teichb{S}$
which admits a measurable subset $E$ of positive linear measure
such that
for any $z_{0}\in E$,
there is a sequence $\{z_{n}\}_{n}\subset \mathbb{D}$
such that
$z_{n}\to z_{0}$ nontangentially and
$d_{T}(x_{0},f_{i}(z_{n}))\to \infty$
as $n\to \infty$
($i=1,2$)
but
$d_{T}(f_{1}(z_{n}),f_{2}(z_{n}))$
remains bounded.
Thus,
Theorem \ref{thm:main} is
recognized as an improvement of Tanigawa's rigidity theorem
of holomorphic families of holomorphic disks in Teichm\"uller space
(cf. \cite[Theorem 1]{Tanigawa}).
The rigidity of holomorphic disks in Teichm\"uller space plays
an important role for studying
holomorphic families of Riemann surfaces
over Riemann surfaces
(cf. \cite{ImayoshiShiga},
\cite{Shiga2}
and \cite{Shiga1}).
We will prove
Theorem \ref{thm:main} in \S\ref{sec:proof-of-the-theorem}.
Applying the rigidity theorem,
we also obtain a uniqueness theorem of holomorphic disks
(cf. Corollary \ref{coro:uniqueness-theorem}).

\subsection{}
\label{subsec:teichmuller-dimension-one}
We first sketch the proof of Theorem \ref{thm:main}
in the case of $\dim_{\mathbb{C}}\Teichb{S}=1$.
Namely,
$S$ is assumed to be either a once holed torus or a fourth holed sphere:
We realize $\Teichb{S}$ in $\mathbb{C}$
via the Bers embedding.
Then
$\Teichb{S}$ is a bounded domain
which is conformally equivalent to the unit disk $\mathbb{D}$
and hence $(\Teichb{S},d_T)$ is isometric to the Poincar\'e hyperbolic disk of curvature $-4$.
Since
the closure of $\Teichb{S}$ 
is homeomorphic to a Jordan domain (cf. \cite{Minsky}),
the Gromov boundary of $(\Teichb{S},d_T)$ is canonically identified with
the Euclidean boundary of $\Teichb{S}$ in $\mathbb{C}$
(cf. \S\ref{subsec:Gromovhyperbolic} below).

By Fatou's theorem,
we may assume that each $f_i$ has non-tangential limit $f^*_i$ at any point of $E$ for $i=1,2$
(cf. \cite[Theorem IV.7]{Tsuji}).
Let $z_{0}\in E$ and take $\{z_{n}\}_{n=1}^{\infty}\subset \mathbb{D}$
as in the theorem.
The condition $\gromov{f_{1}(z_{n})}{f_{2}(z_{n})}{x_{0}}\to \infty$ implies that
$\{f_{1}(z_{n})\}_{n=1}^\infty$
and $\{f_{2}(z_{n})\}_{n=1}^\infty$
determine the same ideal boundary point in the Gromov boundary of $\Teichb{S}$
and hence $f^*_1(z_0)=f^*_2(z_0)$.
Therefore,
we conclude $f_1(z)=f_2(z)$ for all $z\in \mathbb{D}$ by 
Lusin-Priwaloff-Riesz's theorem
(cf. \cite[\S14, \S15]{Lusin-Priwaloff} and \cite[Theorem IV. 9]{Tsuji}).

The proof of the case $\dim_{\mathbb{C}}\Teichb{S}\ge 2$
is established by the similar argument.
Unfortunately,
the situation drastically changes from the above case.
Indeed,
when $\dim_{\mathbb{C}}\Teichb{S}\ge 2$,
Teichm\"uller space is not Gromov hyperbolic,
and
less information is known about the geometry of the Bers boundary
(to the author's knowledge).
To overcome these difficulties,
we will apply the extremal length geometry of Teichm\"uller space
and sophisticated technologies from the theory of Kleinian groups.
We recall these briefly in \S\ref{sec:notation}.

\subsection{}
The Teichm\"uller distance coincides with
the Kobayashi distance on 
Teichm\"uller space (cf. \cite{Royden}).
Since the Kobayashi distances are biholomorphic invariants of complex manifolds,
the rigidity of holomorphic disks stated in Theorem \ref{thm:main}
is thought of as a property of complex manifolds. 
We will observe that
the rigidity property
in our sense
is valid for complex manifolds which are biholomorphic to
bounded strictly pseudoconvex domains with $C^{2}$-boundaries.
Meanwhile,
Teichm\"uller space
is not biholomorphic to such domains unless the complex dimension is one.
The rigidity property does not hold for product manifolds.
As a corollary,
we conclude that Teichm\"uller space is not
realized as
the product of complex manifolds,
which was already proven by H. Tanigawa
(cf. \cite[Corollary 3]{Tanigawa}).

\subsection*{Acknowledgements}
The author thanks Professor Ken'ichi Ohshika
for the stimulating discussions.
He also thanks the referee for his/her useful comments.

\section{Notation}
\label{sec:notation}
\subsection{Teichm\"uller space}
A \emph{marked Riemann surface} is a pair $(X,f)$ of a Riemann surface $X$ of analytically finite type
and an orientation preserving homeomorphism $f\colon {\rm Int}(S)\to X$,
where ${\rm Int}(S)$ is the interior of $S$.
Two marked Riemann surfaces $(X_1,f_1)$ and $(X_2,f_2)$ are said to be \emph{Teichm\"uller equivalent} if there is a conformal mapping $h\colon X_1\to X_2$ such that $h\circ f_1$ is homotopic to $f_2$.
The Teichm\"uller space $\Teichb{S}$ of $S$ is the set of Teichm\"uller equivalence classes of
marked Riemann surfaces.
The \emph{Teichm\"uller distance} is a distance on $\Teichb{S}$ defined by
$$
d_T(x,y)=\frac{1}{2}\inf_h\log K(h)
$$
for $x=(X,f)$ and $y=(Y,g)$ in $\Teichb{S}$,
where $h$ runs all quasiconformal mappings $h\colon X\to Y$ which homotopic to $f_2\circ f_1^{-1}$
and $K(h)$ is the maximal dilatation of $h$.

\subsection{Thurston theory}
\label{subset:thurston-theory}
\subsubsection{Measured laminations}
Let $\mathcal{S}$ be the set of homotopy classes of
non-trivial and non-peripheral simple closed curves on $S$.
Let $\mathcal{WS}$ be the set of weighted simple closed curves $t\alpha$ on $S$,
where $t\ge 0$ and $\alpha\in \mathcal{S}$.
The closure $\mathcal{MF}$
of the image of the embedding
$$
\mathcal{WS}\ni t\alpha\mapsto [\mathcal{S}\ni \beta\mapsto t\cdot i(\alpha,\beta)]
\in \mathcal{R}:=[0,\infty)^\mathcal{S}
$$
is called the \emph{space of measured foliations} on $S$,
where $i(\alpha,\beta)$ is the geometric intersection number between $\alpha$
and $\beta$.
When we fix a complete hyperbolic structure on ${\rm Int}(S)$
of finite area,
measured foliations are canonically identified with \emph{measured geodesic laminations}.
A \emph{geodesic lamination} is a compact set in ${\rm Int}(S)$
which is foliated by disjoint complete geodesics.
A \emph{measured geodesic lamination} is a geodesic lamination
with transverse invariant measure
(e.g. \cite{Bonahon_StonyBrook} and \cite{PeH}).
The underlying geodesic lamination
is called the \emph{support}.

By definition,
$\mathcal{MF}$ contains $\mathcal{WS}$ as a dense subset.
We define $i(t\alpha,s\beta)=ts\,i(\alpha,\beta)$ for $t\alpha,s\beta\in \mathcal{WS}$.
It is known that the intersection number function on $\mathcal{WS}\times \mathcal{WS}$ extends continuously
to the product space $\mathcal{MF}\times \mathcal{MF}$.
The space $\mathcal{R}$ admits a natural action of positive numbers by multiplication.
The quotient space
of $\mathcal{R}-\{0\}$ under this action
is denoted by
$\mathcal{PR}$.
Let $\proj\colon \mathcal{R}-\{0\}\to \mathcal{PR}$
be the projection.
The image $\mathcal{PMF}$ of $\mathcal{MF}-\{0\}$ 
under the projection
is called
the \emph{space of projective measured foliations} on $S$.

\subsubsection{Kleinian groups}
\label{subsubsec:BersEmbedding-endinglamination}
A \emph{Kleinian group} is a discrete subgroup of ${\rm PSL}_{2}(\mathbb{C})$.
Any Kleinian group acts on the hyperbolic $3$-space discontinuously.
By a \emph{Kleinian surface group}
we mean a Kleinian group isomorphic to $\pi_{1}(S)$ via
a type-preserving representation
(i.e. a representation which sends all peripheral loops to parabolic transformations).
An \emph{accidental parabolic transformation} (APT) in  a Kleinian surface group
is a parabolic element which corresponds to a non-peripheral loop on $S$.

Bonahon's tameness theorem asserts that the quotient hyperbolic manifold
of a Kleinian surface group is homeomorphic to ${\rm Int}(S)\times \mathbb{R}$
(cf. \cite{Bonahon_bouts}).
When a Kleinian surface group does not contain APT,
the quotient manifold
has two \emph{ends} corresponding to ${\rm Int}(S)\times \{t>0\}$ and ${\rm Int}(S)\times \{t<0\}$.
An end is said to be \emph{geometrically infinite} or \emph{simply degenerate}
if any neighborhood of the end contains a closed geodesic
which is homotopic to a simple closed curve on ${\rm Int}(S)\times \{0\}$.
For a geometrically infinite end,
we associate a unique geodesic lamination,
which we call the \emph{ending lamination} of the geometrically infinite end.
The ending lamination is \emph{filling} in the sense that 
it intersects transversely the support of every measured lamination except for itself
(cf. \cite{Bonahon_bouts}, \cite[\S2.5]{BCM}  and \cite{Thurston}).

A \emph{quasifuchsian group} is, by definition,
a Kleinian surface group which is obtained by a quasiconformal deformation of a Fuchsian group.
A Kleinian surface group is said to be
\emph{b-group} if
it has a unique simply connected invariant component.
A b-group is called a \emph{totally degenerate group} if 
its region of discontinuity is connected
(cf. \cite[\S2]{Bers}).

\subsubsection{The Bers embedding and b-groups}
\label{subsubsec:BersEmbedding}
The Teichm\"uller space $\Teichb{S}$ of $S$
is embedded into a finite dimensional complex Banach space
via the Bers embedding (cf. \cite[\S1]{Bers}).
The image of the Bers embedding is a bounded domain.
By the Bers embedding,
each point in $\Teichb{S}$ is associated with a quasifuchsian group.
Every point in the boundary of $\Teichb{S}$, called the \emph{Bers boundary},
corresponds to a b-group.
If a totally degenerate group does not contain APT,
the quotient hyperbolic manifold has a unique geometrically infinite end.
The \emph{ending lamination theorem} asserts that
two totally degenerate groups without APT in the Bers boundary agree
if and only if they have the same ending lamination
(cf. \cite{BCM}).

\subsection{Extremal length geometry of Teichm\"uller space}
\subsubsection{Gromov product of the Teichm\"uller distance}
For $\alpha\in \mathcal{S}$ and $y=(Y,f)\in \Teichb{S}$
we denote by $\ext_y(\alpha)$ the \emph{extremal length} of
the family of rectifiable simple closed curves on $Y$ homotopic to $f(\alpha)$.
When we put $\ext_{y}(t\alpha)=t^{2}\ext_{y}(\alpha)$,
the extremal length
extends continuously to $\mathcal{MF}$
(cf. \cite[Proposition 3]{Ker}).
The \emph{Gardiner-Masur embedding} $\Phi_{GM}$ is defined by
$$
\Phi_{GM}\colon
\Teichb{S}\ni y\mapsto \proj([\mathcal{S}\ni \alpha\mapsto \ext_{y}(\alpha)^{1/2}])\in \mathcal{PR}
$$
The closure $\cl{\Teichb{S}}$ of the image is called the \emph{Gardiner-Masur closure}
and the complement $\partialGM{\Teichb{S}}= \cl{\Teichb{S}}-\Phi_{GM}(\Teichb{S})$
the \emph{Gardiner-Masur boundary}.
F. Gardiner and H. Masur observed that 
the closure $\cl{\Teichb{S}}$ is compact and
$\mathcal{PMF}\subset \partialGM{\Teichb{S}}$
(cf. \cite{GM})

In \cite{Miyachi1},
the author proved the following theorem.

\begin{theorem}[Extension theorem]
\label{thm:extension}
Fix $x_{0}\in \Teichb{S}$.
The Gromov product $\gromov{\,\cdot\,}{\,\cdot\,}{x_{0}}$
on $\Teichb{S}\times \Teichb{S}$ extends continuously 
to $\cl{\Teichb{S}}\times \cl{\Teichb{S}}$
with values in the interval $[0,\infty]$.
Furthermore,
for $[F],[G]\in \mathcal{PMF}\subset \partialGM{\Teichb{S}}$,
we have
\begin{equation}
\label{eq:intersection1}
\exp(-2\gromov{[F]}{[G]}{x_{0}})=
\frac{i(F,G)}{\ext_{x_{0}}(F)^{{1/2}}\ext_{x_{0}}(G)^{{1/2}}}.
\end{equation}
\end{theorem}

\subsubsection{Intersection number with basepoint}
We define the \emph{intersection number with basepoint} $x_{0}\in \Teichb{S}$ by
$$
i_{x_{0}}(p_{1},p_{2})=\exp(-2\gromov{p_{1}}{p_{2}}{x_{0}})
$$
for $p_{1},p_{2}\in \cl{\Teichb{S}}$.
It is known that
\begin{equation}
\label{eq:intersection2}
i_{x_{0}}(y,[F])=
\frac{e^{-d_{T}(x_{0},y)}\ext_{y}(F)^{1/2}}{\ext_{x_{0}}(F)^{1/2}
}
\end{equation}
for $y\in \Teichb{S}$ and $[F]\in \mathcal{PMF}$
where we set $\exp(-\infty)=0$
(cf. \cite[\S5.1]{Miyachi1}).
For $p\in \cl{\Teichb{S}}$,
we define
$$
\mathcal{N}(p)=\{q\in \cl{\Teichb{S}}\mid i_{x_{0}}(p,q)=0\}.
$$
In \cite{Miyachi2},
the author showed the following.

\begin{theorem}[Null set]
\label{thm:null-set}
$\mathcal{N}(p)\ne \emptyset$ if and only if $p\in \partialGM{\Teichb{S}}$.
In addition,
for any $p\in \partialGM{\Teichb{S}}$,
there is $[F]\in \mathcal{PMF}$ such that
$\mathcal{N}(p)=\mathcal{N}([F])$.
\end{theorem}

\subsection{Gromov hyperbolic space}
\label{subsec:Gromovhyperbolic}
Let $(X,d_{X})$ be a metric space.
Let $x_{0}\in X$ be a basepoint.
The \emph{Gromov product} with reference point $x_{0}$ is defined by
\begin{equation}
\label{eq:Gromov-product-X}
\gromov{x}{y}{x_{0}}^{X}=\frac{1}{2}(d_{X}(x_{0},x)+d_{X}(x_{0},y)-d_{X}(x,y)).
\end{equation}
A \emph{Gromov hyperbolic space}
is a metric space $(X,d_{X})$ with the property that
there is $\delta>0$ such that
$$
\gromov{x}{y}{x_{0}}^{X}\ge \min\{\gromov{x}{z}{x_{0}}^{X},\gromov{y}{z}{x_{0}}^{X}\}-\delta
$$
for all $x,y,z\in X$
(\cite[\S1.1]{Gromov}).

Let $(X,d_{X})$ be a Gromov hyperbolic space.
A sequence $\{x_{n}\}_{n=1}^{\infty}\subset X$ is said to be \emph{convergent at infinity} if
$\gromov{x_{n}}{x_{m}}{x_{0}}\to \infty$ as $n,m\to \infty$.
Two convergent sequences $\{x_{n}\}_{n=1}^{\infty}$ and $\{y_{n}\}_{n=1}^{\infty}$ at infinity
are \emph{equivalent} if $\liminf_{n\to 0}\gromov{x_{n}}{y_{n}}{x_{0}}=\infty$.
The set of equivalence classes of convergent sequences at infinity is called the
\emph{Gromov boundary} and denoted by $\partial_{\infty}X$
(\cite[\S1.8]{Gromov}).
The Poincar\'e hyperbolic disk $(\mathbb{D},d_{\mathbb{D}})$
is a typical example of Gromov hyperbolic space.
The Gromov boundary $\partial_{\infty}\mathbb{D}$
of $(\mathbb{D},d_{\mathbb{D}})$ is canonically identified with
the Euclidean boundary $\partial \mathbb{D}$
(\cite[\S1.5]{Gromov}).
However,
when $\dim_{\mathbb{C}}\Teichb{S}\ge 2$,
$(\Teichb{S},d_{T})$ is not Gromov hyperbolic
(cf. \cite[Theorem 3.1]{MW2}).

\section{Proof of the theorem}
\label{sec:proof-of-the-theorem}
\subsection{Proof of the theorem}
We identify $\Teichb{S}$ with a bounded domain in a finite dimensional
complex Banach space
via the Bers embedding.
By Fatou's theorem,
there is a measurable set $E_{0}\subset \partial \mathbb{D}$ of full measure
such that
$f_{1}$ and $f_{2}$ has nontangential limits at every $z_{0}\in E_{0}$.
Furthermore,
from Shiga's theorem
\cite[Theorem 5]{Shiga2},
we may assume that the nontangential limit at any point in $E_{0}$
corresponds to either a quasifuchsian group or a totally degenerate group
without APT.

Let $E_{1}=E_{0}\cap E$ and $z_{0}\in E_{1}$.
By the assumption,
there is a sequence $\{z_{n}\}_{n=1}^{\infty}\subset \mathbb{D}$
such that $z_{n}\to z_{0}$ nontangentially
and $\gromov{f_{1}(z_{n})}{f_{2}(z_{n})}{x_{0}}\to \infty$
as $n\to \infty$.
Denote by $f_{i}^{*}(z_{0})$
the nontangential limit of $f_{i}$ at $z_{0}$.
Since $d_{T}(x_{0},f_{i}(z_{n}))\to \infty$,
$f_{i}^{*}(z_{0})$ corresponds to a totally degenerate group
for $i=1,2$.
Let $\lambda_{1}$ and $\lambda_{2}$ be the ending laminations of
geometrically infinite ends of
the hyperbolic manifolds associated with
$f_{1}^{*}(z_{0})$ and $f_{2}^{*}(z_{0})$.

Fix $i=1,2$.
Take $\alpha^{i}_{n}\in \mathcal{S}$ with
$\ext_{f_{i}(z_{n})}(\alpha^{i}_{n})\le M$
for some constant $M>0$ independent of $n$
(cf. \cite[Theorem 1]{Bers2}).
By taking a subsequence,
there is a bounded sequence $\{t^{i}_{n}\}_{n}$
such that $t^{i}_{n}\alpha^{i}_{n}\to \mu_{i}\in \mathcal{MF}-\{0\}$.
Since $f_{i}(z_{n})$ converges to a totally degenerate group without APT,
from \cite[Theorem 2]{Abikoff},
we can see that $\ext_{x_{0}}(\alpha^{i}_{n})\to \infty$ as $n\to\infty$.
Hence,
we have that $t^{i}_{n}\to 0$
since $(t^{i}_{n})^{2}\ext_{x_{0}}(\alpha^{i}_{n})\to \ext_{x_{0}}(\mu_{i})$.
By Bers' inequality \cite[Theorem 3]{Bers}
and Maskit's comparison theorem \cite{Maskit},
the hyperbolic length of $t^{i}_{n}\alpha^{i}_{n}$ in the quasifuchsian manifold
associated with $f_{i}(z_{n})$ tends to $0$.
From the continuity of the Thurston's length function,
any sublamination of the support of $\mu_{i}$ is non-realizable
in the hyperbolic manifold associated with $f_{i}^{*}(z_{0})$
(cf. \cite{Ohshika} and \cite[Theorem 7.1, Corollary 7.3]{Brock}).
Hence,
the support of $\mu_{i}$
is contained in $\lambda_{i}$
(cf. \cite{Bonahon_bouts} and \cite[\S9]{Thurston}).
Since $\lambda_{i}$ is filling on $S$,
the support of $\mu_{i}$ coincides with $\lambda_{i}$
(cf. \S\ref{subsubsec:BersEmbedding-endinglamination}).

By taking a subsequence if necessary,
we may assume that $\{\Phi_{GM}(f_{i}(z_{n}))\}_{n=1}^{\infty}$
converges to a point $p_{i}\in\partialGM{\Teichb{S}}$.
By Theorem \ref{thm:null-set},
there is $\nu_{i}\in \mathcal{MF}$
such that $\mathcal{N}(p_{i})=\mathcal{N}([\nu_{i}])$.
By Theorem \ref{thm:extension} and \eqref{eq:intersection2},
we have
\begin{align*}
i_{x_{0}}(p_{i},[\mu_{i}])
&=
\lim_{n\to \infty}
i_{x_{0}}(f_{i}(z_{n}),t^{i}_{n}\alpha^{i}_{n}) \\
&=
\lim_{n\to \infty}e^{-d_{T}(x_{0},f_{i}(z_{n}))}
\frac{\ext_{f_{i}(z_{n})}(t^{i}_{n}\alpha^{i}_{n})^{1/2}}
{\ext_{x_{0}}(t^{i}_{n}\alpha^{i}_{n})^{1/2}} \\
&\le \lim_{n\to \infty}\frac{M^{1/2}t^{i}_{n}e^{-d_{T}(x_{0},f_{i}(z_{n}))}}{{\ext_{x_{0}}(t^{i}_{n}\alpha^{i}_{n})^{1/2}}}
=0.
\end{align*}
Hence
we obtain
$i(\nu_{i},\mu_{i})=0$ from \eqref{eq:intersection1}.
Therefore,
the support of $\nu_{i}$ coincides with that of $\mu_{i}$
since the ending lamination $\lambda_{i}$ is filling
(cf. \S\ref{subsubsec:BersEmbedding-endinglamination}).

Our assumption $\gromov{f_{1}(z_{n})}{f_{2}(z_{n})}{x_{0}}\to \infty$
implies that $i_{x_{0}}(p_{1},p_{2})=0$
and hence 
 $i(\nu_{1},\nu_{2})=0$
from \eqref{eq:intersection1} again.
Thus we obtain that $\lambda_{1}=\lambda_{2}$
and $f_{1}^{*}(z_{0})=f_{2}^{*}(z_{0})$
 from the ending lamination theorem
 (cf. \S\ref{subsubsec:BersEmbedding}).
Since $E_{1}$ has positive linear measure,
the coincidence between $f_{1}$ and $f_{2}$ on $\mathbb{D}$
follows from Lusin-Priwaloff-Riesz's theorem.

\subsection{Uniqueness of holomorphic disks}
From Theorem \ref{thm:main},
we conclude the following uniqueness theorem.

\begin{corollary}[Uniqueness theorem]
\label{coro:uniqueness-theorem}
Let $f_{1},f_{2}\colon \mathbb{D}\to \Teichb{S}$ be holomorphic mappings.
The following are equivalent.
\begin{enumerate}
\item
$f_{1}(z)=f_{2}(z)$ for all $z\in \mathbb{D}$.
\item
There is a measurable subset $E\subset \partial \mathbb{D}$ of positive linear measure
such that for any $z_{0}\in E$ there is a sequence $\{z_{n}\}_{n=1}^{\infty}\subset \mathbb{D}$ converging
nontangentially to $z_{0}$ which satisfies one of the following:
\begin{enumerate}
\item
$\gromov{f_{1}(z_{n})}{f_{2}(z_{n})}{x_{0}}=O(1)$ and $d_{T}(f_{1}(z_{n}),f_{2}(z_{n}))\to 0$
as $n\to \infty$.
\item
$\gromov{f_{1}(z_{n})}{f_{2}(z_{n})}{x_{0}}\to \infty$ as $n\to \infty$.
\end{enumerate}
\end{enumerate}
\end{corollary}

\begin{proof}
We only check that  (2) implies (1).
Suppose the assertion (2).
We realize $\Teichb{S}$ as a bounded domain via the Bers embedding.
From Shiga's theorem,
we may assume that each $f_{i}$ has the non-tangential limit $f^{*}_{i}$
at any point in $E$ and the limit corresponds to
either a quasifuchsian group or a totally degenerate group without APT.

Let $z_{0}\in E$ and take a sequence $\{z_{n}\}_{n=1}^{\infty}\subset \mathbb{D}$
as in the assertion (2).
Suppose (a) holds.
Since
$$
d_{T}(x_{0},f_{i}(z_{n}))
\le
2\gromov{f_{1}(z_{n})}{f_{2}(z_{n})}{x_{0}}+d_{T}(f_{1}(z_{n}),f_{2}(z_{n}))
=O(1)
$$
as $n\to \infty$ for $i=1,2$,
the limits $f^{*}_{1}(z_{0})$ and $f^{*}_{2}(z_{0})$ are quasifuchsian groups.
Since $d_{T}(f_{1}(z_{n}),f_{2}(z_{n}))\to 0$,
we have
$f^{*}_{1}(z_{0})=f^{*}_{2}(z_{0})$.
If (b) holds,
we also deduce the equality $f^{*}_{1}(z_{0})=f^{*}_{2}(z_{0})$
by the same argument as that in Theorem \ref{thm:main}.
\end{proof}

\section{Rigidity of holomorphic disks in complex manifolds}
\label{sec:concluding-remark}
We shall discuss what kind of
complex manifolds
the rigidity theorem in our sense is valid.
Henceforth,
let $\Omega$ be a complex manifold.
Denote by $d_{\Omega}$ the Kobayashi distance on $\Omega$.
Fix a point $x_{0}\in \Omega$
and set
$\gromov{x}{y}{x_0}^\Omega$
to be the Gromov product on $(\Omega,d_{\Omega})$
with reference point $x_{0}$
(cf. \eqref{eq:Gromov-product-X}).

\subsection{}
The rigidity theorem in our sense
holds when
$\Omega$ is biholomorphic to
a bounded strongly pseudoconvex domain
with $C^{2}$-boundary:
The proof is established by the same argument as that for the case of Teichm\"uller space
of dimension one
(cf. \S\ref{subsec:teichmuller-dimension-one}).
However,
we shall give a proof for the completeness.

Notice that I. Graham showed that
$(\Omega,d_{\Omega})$ is a complete metric space
(cf \cite[Proposition 5]{Graham}).
In addition,
Z. Balogh and M. Bonk
observed that $(\Omega,d_{\Omega})$ is Gromov hyperbolic 
and the Gromov boundary $\partial_{\infty}\Omega$ of
$\Omega$ canonically coincides with the Euclidean boundary
$\partial\Omega$
(cf. \cite[Theorem 1.4]{BB}).

Let $f$ and $g$ be holomorphic mappings from $\mathbb{D}$ to $\Omega$.
Suppose that there is a measurable set $E\subset \partial \mathbb{D}$
of positive linear measure
such that 
for any $z_{0}\in E$,
there is a sequence $\{z_{n}\}_{n=1}^{\infty}\subset \mathbb{D}$
such that $z_{n}\to z_{0}$ nontangentially
and $\gromov{f(z_{n})}{g(z_{n})}{x_{0}}^{\Omega}\to \infty$.
Since $\Omega$ is a bounded domain,
we may assume that each of $f$ and $g$ admits
the nontangential limit at every point in $E$.
The condition $\gromov{f(z_{n})}{g(z_{n})}{x_{0}}^\Omega\to \infty$
implies that sequences
$\{f(z_{n})\}_{n=1}^\infty$
and $\{g(z_{n})\}_{n=1}^\infty$
converge the same ideal boundary point in $\partial_{\infty}\Omega=\partial\Omega$
(cf. \S\ref{subsec:Gromovhyperbolic}).
Hence $f$ and $g$ have the same nontangential limits on $E$.
Since $E$ has positive linear measure,
$f$ coincides with $g$ on $\mathbb{D}$ by Lusin-Priwaloff-Riesz's theorem
as in the previous section.

\subsection{}
On the other hand,
when $\dim_\mathbb{C}\Teichb{S}\ge 2$,
$(\Teichb{S},d_{T})$ is not Gromov hyperbolic.
Hence
the argument in the previous section does not work for Teichm\"uller spaces
unless $\dim_\mathbb{C}\Teichb{S}=1$.
As a consequence,
the class of complex manifolds
with the rigidity property in our sense is strictly
larger than the class of
bounded Gromov-hyperbolic pseudoconvex domains (in terms of the Kobayashi distances)
whose Gromov boundaries coincide with the Euclidean boundaries.
By applying the discussion in the previous section,
we can easily see that any pseudoconvex domain in the latter class 
satisfies the rigidity property in our sense.
One can also check that the uniqueness theorem in our sense also holds
for domains in the latter class.

\subsection{}
The rigidity theorem in our sense does not hold
if $\Omega$ is biholomorphic to
the product manifold $M_{1}\times M_{2}$
of some complex manifolds $M_{i}$ ($i=1,2$)
which admits a holomorphic mapping $f\colon \mathbb{D}\to \Omega$
with the property that there is a measurable set $E\subset \partial \mathbb{D}$
of positive linear measure such that
for any $z_{0}\in E$
there is a sequence $\{z_{n}\}_{n=1}^{\infty}\subset \mathbb{D}$
such that $z_{n}\to z_{0}$ nontangentially
and $d_{\Omega}(x_{0},f(z_{n}))\to \infty$.
For instance,
when $M_{2}=\mathbb{D}$,
a product manifold $M_{1}\times M_{2}$ has this property.
However,
when each $M_{i}$ is a closed complex manifold,
the product manifold $M_{1}\times M_{2}$ does not have the property.

It is known that
\begin{align}
\max\{d_{M_{1}}(z^{1},z^{2}),
d_{M_{2}}(w^{1},w^{2})\}
&\le
d_{\Omega}((z^{1},z^{2}),(w^{1},w^{2})) 
\nonumber\\
&\le
d_{M_{1}}(z^{1},z^{2}),
+d_{M_{2}}(w^{1},w^{2})
\label{eq:product1}
\end{align}
for $(z^{1},z^{2}),(w^{1},w^{2})\in \Omega=M_{1}\times M_{2}$
(cf. \cite[Proposition 2.5]{Kobayashi}).

Let $f=(f_{1},f_{2})$ and $x_{0}=(x^{1}_{0},x^{2}_{0})$.
From \eqref{eq:product1},
by taking a measurable subset in $E$ of positive linear measure
if necessary,
we may assume that for any $z_{0}\in E$
there is a sequence $\{z_{n}\}_{n=1}^{\infty}\subset \mathbb{D}$
such that $z_{n}\to z_{0}$ nontangentially
and $d_{M_{1}}(x^{1}_{0},f_{1}(z_{n}))$ tends to $\infty$.
Let $y^{2}_{0}\in M_{2}$ with $y^{2}_{0}\ne x^{2}_{0}$.
Define
\begin{align*}
g_{1}(z)&=(f_{1}(z),x^{2}_{0}) \\
g_{2}(z)&=(f_{1}(z),y^{2}_{0}).
\end{align*}
Then, $g_{1}(z)\ne g_{2}(z)$
but
$
d_{\Omega}(g_{1}(z),g_{2}(z))=d_{M_{2}}(x^{2}_{0},y^{2}_{0})
$
for all $z\in \mathbb{D}$.
For any $z_{0}\in E$,
there is a sequence $\{z_{n}\}_{n}\subset \mathbb{D}$
such that $z_{n}\to z_{0}$ nontangentially and
\begin{align*}
\gromov{g_{1}(z_{n})}{g_{2}(z_{n})}{x_{0}}^\Omega
&=
\frac{1}{2}(d_{\Omega}(x_{0},g_{1}(z_{n}))+d_{\Omega}(x_{0},g_{2}(z_{n}))-
d_{\Omega}(g_{1}(z_{n}),g_{2}(z_{n}))) \\
&\ge 
d_{M_{1}}(x^{1}_{0},f_{1}(z_{n}))-d_{M_{2}}(x^{2}_{0},y^{2}_{0})\to\infty
\end{align*}
as $n\to \infty$.
%


\end{document}